\RequirePackage{fix-cm}
\documentclass[smallcondensed,envcountsame]{svjour3} 
\smartqed  
\usepackage{graphicx, amssymb}
\usepackage{times,a4wide,mathrsfs}
\usepackage{amsmath}
\usepackage{amsfonts}

\newcommand{\C}{\mathbb{C}}
\newcommand{\ZZ}{\mathbb{Z}}

\newcommand{\QQ}{\mathbb{Q}}
\newcommand{\NN}{\mathbb{N}}
\newcommand{\PP}{\mathbb{P}}

\newcommand{\MM}{\mathcal M}

\newcommand{\wt}{\widetilde}
\newcommand{\ima}{\hbox{Im}}
\newcommand{\rom}{\romannumeral}

\newtheorem{convention}{Conventions}

\newtheorem{nonumbering}{Theorem}

\begin{document}

\title{Some desultory remarks concerning algebraic cycles and Calabi--Yau threefolds}

\author{Robert Laterveer}

\institute{CNRS - IRMA, Universit\'e de Strasbourg \at
              7 rue Ren\'e Descartes \\
              67084 Strasbourg cedex\\
              France\\
              \email{laterv@math.unistra.fr}   }

\date{Received: date / Accepted: date}

\maketitle

\begin{abstract} We study some conjectures about Chow groups of varieties of geometric genus one. Some examples are given of Calabi--Yau threefolds where these conjectures can be verified, using the theory of finite--dimensional motives.
\end{abstract}



\section{Introduction}

This note is about some specific questions concerning Chow groups $A^\ast X$ of complex varieties.
In this field, the following relative version of Bloch's conjecture occupies a central position:

\begin{conjecture}[Bloch \cite{B}]\label{relbloch} Let $X$ be a smooth projective variety over $\C$ of dimension $n$. Let $\Gamma\in A^n(X\times X)$ be a correspondence such that
  \[ \Gamma_\ast=\hbox{id}\colon\ \  H^{j,0}(X)\ \to\ H^{j,0}(X)\ \ \hbox{for\ all\ } j=2,\ldots, n.\]
  Then
  \[ \Gamma_\ast=\hbox{id}\colon\ \ A^n_{AJ}(X)\ \to\ A^n_{AJ}(X)\ .\]
  (Here $A^n_{AJ}(X)$ denotes the subgroup of $0$--cycles in the kernel of the Albanese map.) 
  \end{conjecture}
  
  This conjecture is open in most interesting cases.
  
 A second conjecture adressed in this note is specific to varieties with $p_g=1$:
 
  \begin{conjecture}\label{indecomp} Let $X$ be a smooth projective variety over $\C$ of dimension $n$, and $p_g(X)=1$. Then there exists a ``transcendental motive'' $t(X)\in\MM_{\rm rat}$, responsible for $H^{n,0}(X)$, which is indecomposable: any submotive of $t(X)$ is either $0$ or $t(X)$.
  \end{conjecture}
  
  This is motivated by results of Voisin \cite{V8} and Pedrini \cite{Ped}, who prove that for certain $K3$ surfaces, the transcendental motive $t_2(X)$ (in the sense of \cite{KMP}) is indecomposable.
  
  A third conjecture adressed in this note is the following conjecture made by Voisin concerning self--products of varieties of geometric genus one. (For simplicity, we only state the
  conjecture in the case of odd dimension.)
  
  \begin{conjecture}[Voisin \cite{V9}]\label{conjvois} Let $X$ be a smooth projective variety of odd dimension $n$, with $p_g(X)=1$ and $h^{j,0}(X)=0$ for $0<j<n$. For any $k\ge 2$, let the symmetric group $S_k$ act on $X^k$ by permutation of the factors.
Let $pr_k\colon X^k\to X^{k-1}$ denote the projection obtained by omitting one of the factors. Then the induced map 
  \[  (pr_k)_\ast\colon\ \ A_j(X^k)^{S_k}\ \to\ A_j(X^{k-1})\]
  is injective for $j\le k-2$.
\end{conjecture}

In this note, using elementary arguments, some examples are given of Calabi--Yau threefolds where these conjectures are verified. The following is a sample of this (slightly more general statements can be found below):

\begin{nonumbering}[=Theorems \ref{main2}, \ref{main3} and \ref{main1}] Let $X$ be a Calabi--Yau threefold which is rationally dominated by a product of elliptic curves. 
Then conjecture \ref{indecomp} and a weak form of conjecture \ref{relbloch} are true for $X$. If in addition $h^{2,1}(X)=0$, then conjecture \ref{conjvois} is true for $X$.
\end{nonumbering}

One example where this applies is Beauville's threefold \cite{Beau0}; other examples are given below. The main tool used in this note is the theory of finite--dimensional motives of Kimura and O'Sullivan \cite{Kim}.

\vskip0.4cm

\begin{convention} All varieties will be projective irreducible varieties over $\C$.

For smooth $X$, we will denote by $A^j(X)$ the Chow group $CH^j(X)\otimes{\QQ}$ of codimension $j$ cycles under rational equivalence. The notations $A^j_{hom}(X)$ and $A^j_{AJ}(X)$ will denote the subgroup of homologically trivial and Abel--Jacobi trivial cycles respectively. $\MM_{\rm rat}$ will denote the (contravariant) category of Chow motives with $\QQ$--coefficients over $\C$.
For a smooth projective variety over $\C$, $h(X)=(X,\Delta_X,0)$ will denote its motive in $\MM_{\rm rat}$. $H^\ast(X)$ will denote singular cohomology with $\QQ$--coefficients.





\end{convention}

\section{Some Calabi--Yau threefolds}
\label{examples}

This section presents some examples of Calabi--Yau threefolds to which our arguments apply.

\begin{definition}[Calabi--Yau]\label{CY} In this note, a smooth projective variety $X$ of dimension $3$ is called a {\em Calabi--Yau threefold\/} if $h^{3,0}(X)=1$ and $h^{1,0}(X)=h^{2,0}(X)=0$.
\end{definition}

\begin{remark} Definition \ref{CY} is non--standard; usually, one requires that the canonical bundle is trivial. For the purposes of the present note, however, definition \ref{CY} suffices.
\end{remark}

\subsection{Rigid examples}
\label{rigid}

\begin{example}[Beauville \cite{Beau0}, Strominger--Witten \cite{SW}] Let $E$ be the Fermat elliptic curve, and let $\varphi\colon E\to E$ be the automorphism given by $(x,y,z)\mapsto (x,y, \zeta z)$, where $\zeta$ is a primitive third root of unity. Let 
  \[ \varphi_3=\varphi\times \varphi\times \varphi\colon\ \ E^3\ \to\ E^3\]
be the automorphism acting as $\varphi$ on each factor. Let $\wt{E^3}\to E^3$ denote the blow--up of the 27 fixed points of $\varphi^3$, and let
\[ \wt{\varphi_3}\colon\ \ \wt{E^3}\ \to\ \wt{E^3}\]
denote the automorphism induced by $\varphi^3$. The quotient
\[ Z:= \wt{E^3}/ \wt{\varphi_3}\]
is a smooth Calabi--Yau threefold, which is rigid (i.e. $h^{2,1}(Z)=0$).  
\end{example}

\begin{remark} The threefold $Z$ is relevant in string theory. Indeed, as explained in the nice article \cite{FG} (where $Z$ is studied in great detail), the rigidity of $Z$ posed a conundrum to physicists: the mirror of $Z$ cannot be a projective threefold ! This is discussed in \cite{CHSW}, and led to the subsequent development of a theory of generalized mirror symmetry \cite{CDP}.
\end{remark}

\begin{example}[\cite{GG2}, \cite{R}] The group $G=(\ZZ_3)^2=\langle\zeta\times\zeta\times\zeta, \zeta\times\zeta^2\times 1\rangle$ acts on $E^3$, and there exists a desingularization
  \[ Z_2\ \to\ E^3/G\]
  which is Calabi--Yau. The variety $Z_2$ is rigid \cite{R}.
  \end{example}
  
\subsection{More (not necessarily rigid) examples}

\begin{example}[Oguiso--Sakurai \cite{OS}] The varieties $X_{3,1}$ and $X_{3,2}$ constructed in \cite[Theorem 3.4]{OS} are Calabi--Yau threefolds, obtained as crepant resolutions of quotients $E^3/G$, where $E$ is an elliptic curve and $G\subset\hbox{Aut}(E^3)$ a certain group.\footnote{The definition of Calabi--Yau variety in \cite{OS} is different from ours, as it is not required that $h^{2,0}=0$; however (as noted in \cite[Section 4.1]{FG}, the varieties $X_{3,1}$ and $X_{3,2}$ do have $h^{2,0}=0$.}
\end{example}

\begin{example}[Borcea--Voisin]  Let $S$ be a $K3$ surface admitting a non--symplectic involution $\alpha$ which fixes $k= 10$ rational curves. Let $E$ be an elliptic curve, and let $\iota\colon E\to E$ be the involution $z\mapsto -z$. There exists a desingularization
  \[  X\ \to\ (S\times E)/(\alpha\times\iota)\ \]
  which is Calabi--Yau; it has $h^{2,1}(X)=11-k$ \cite{V20}, \cite{Bo}. 
  
To be sure, the Borcea--Voisin construction exists more generally for any $k\le 10$ \cite{V20}, \cite{Bo}; in this note, however, we only consider the extremal case $k=10$. In this easy case of the Borcea--Voisin construction, the $K3$ surface $S$ is rationally dominated by a product of elliptic curves. Also (as explained in \cite[2.4]{GG}), $X$ is birational to a double cover of $\PP^3$ branched along $8$ planes.
\end{example}
 
\begin{remark}\label{more} In \cite{CG}, the Borcea--Voisin construction is generalized, to include quotients of higher--order automorphisms of $S\times E$. In some cases, e.g. \cite[Table 2 lines 18 and 19]{CG}, the resulting Calabi--Yau threefold is rationally dominated by curves, and is rigid (cf. \cite[Remarks 6.3 and 6.5]{CG}).
\end{remark}

\section{Preliminaries}

\subsection{Standard conjecture $B(X)$}

Let $X$ be a smooth projective variety of dimension $n$, and $h\in H^2(X,\QQ)$ the class of an ample line bundle. The hard Lefschetz theorem asserts that the map
  \[  L^{n-i}\colon H^i(X,\QQ)\to H^{2n-i}(X,\QQ)\]
  obtained by cupping with $h^{n-i}$ is an isomorphism, for any $i< n$. One of the standard conjectures asserts that the inverse isomorphism is algebraic:

\begin{definition} Given a variety $X$, we say that $B(X)$ holds if for all ample $h$, and all $i<n$ the isomorphism 
  \[  (L^{n-i})^{-1}\colon 
  H^{2n-i}(X,\QQ)\stackrel{\cong}{\rightarrow} H^i(X,\QQ)\]
  is induced by a correspondence.
 \end{definition}  
 
 \begin{remark} It is known that $B(X)$ holds for the following varieties: curves, surfaces, abelian varieties \cite{K0}, \cite{K}, threefolds not of general type \cite{Tan}, hyperk\"ahler varieties of 
 $K3^{[n]}$--type \cite{ChM}, $n$--dimensional varieties $X$ which have $A_i(X)_{}$ supported on a subvariety of dimension $i+2$ for all $i\le{n-3\over 2}$ \cite[Theorem 7.1]{V}, $n$--dimensional varieties $X$ which have $H_i(X)=N^{\llcorner {i\over 2}\lrcorner}H_i(X)$ for all $i>n$ \cite[Theorem 4.2]{V2}, products and hyperplane sections of any of these \cite{K0}, \cite{K}.
 
 For smooth projective varieties over $\C$, the standard conjecture $B(X)$ implies the standard conjecture $D(X)$, i.e homological and numerical equivalence coincide on $X$ and 
 $X\times X$ \cite{K0}, \cite{K}. 
 \end{remark}

\subsection{Coniveau and niveau filtration}

\begin{definition}[Coniveau filtration \cite{BO}]\label{con} Let $X$ be a quasi--projective variety. The {\em coniveau filtration\/} on cohomology and on homology is defined as
  \[\begin{split}   N^c H^i(X,\QQ)&= \sum \ima\bigl( H^i_Y(X,\QQ)\to H^i(X,\QQ)\bigr)\ ;\\
                           N^c H_i(X,\QQ)&=\sum \ima \bigl( H_i(Z,\QQ)\to H_i(X,\QQ)\bigr)\ ,\\
                           \end{split}\]
   where $Y$ runs over codimension $\ge c$ subvarieties of $X$, and $Z$ over dimension $\le i-c$ subvarieties.
 \end{definition}

Vial introduced the following variant of the coniveau filtration:

\begin{definition}[Niveau filtration \cite{V4}] Let $X$ be a smooth projective variety. The {\em niveau filtration} on homology is defined as
  \[ \wt{N}^j H_i(X)=\sum_{\Gamma\in A_{i-j}(Z\times X)_{}} \ima\bigl( H_{i-2j}(Z)\to H_i(X)\bigr)\ ,\]
  where the union runs over all smooth projective varieties $Z$ of dimension $i-2j$, and all correspondences $\Gamma\in A_{i-j}(Z\times X)_{}$.
  The niveau filtration on cohomology is defined as
  \[   \wt{N}^c H^iX:=   \wt{N}^{c-i+n} H_{2n-i}X\ .\]
  
\end{definition}

\begin{remark}\label{is}
The niveau filtration is included in the coniveau filtration:
  \[ \wt{N}^j H^i(X)\subset N^j H^i(X)\ .\] 
  These two filtrations are expected to coincide; indeed, Vial shows this is true if and only if the Lefschetz standard conjecture is true for all varieties \cite[Proposition 1.1]{V4}. 
  
  Using the truth of the Lefschetz standard conjecture in degree $\le 1$, it can be checked \cite[page 6 "Properties"]{V4} that the two filtrations coincide in a certain range:
  \[  \wt{N}^j H^i(X)= N^j H^iX\ \ \ \hbox{for\ all\ }j\ge {i-1\over 2} \ .\]
  \end{remark}

\subsection{Finite--dimensional motives}

We refer to \cite{Kim}, \cite{An}, \cite{J4}, \cite{MNP} for basics on finite--dimensional motives. 
A crucial property is the nilpotence theorem, which allows to lift relations between cycles from homological to rational equivalence:

\begin{theorem}[Kimura \cite{Kim}]\label{nilp} Let $X$ be a smooth projective variety of dimension $n$ with finite--dimensional motive. Let $\Gamma\in A^n(X\times X)_{}$ be a correspondence which is numerically trivial. Then there is $N\in\NN$ such that
     \[ \Gamma^{\circ N}=0\ \ \ \ \in A^n(X\times X)_{}\ .\]
\end{theorem}

   Conjecturally, any variety has finite--dimensional motive \cite{Kim}. We are still far from knowing this, but at least there are quite a few non--trivial examples:
 
\begin{remark} 
The following varieties have finite--dimensional motive: abelian varieties, varieties dominated by products of curves \cite{Kim}, $K3$ surfaces with Picard number $19$ or $20$ \cite{P}, surfaces not of general type with $p_g=0$ \cite[Theorem 2.11]{GP}, many examples of surfaces of general type with $p_g=0$ \cite{PW}, \cite{V8}, 
generalized Kummer varieties \cite[Remark 2.9(\rom2)]{Xu},
 3--folds and 4--folds with nef tangent bundle \cite{Iy}, \cite{Iy2}, 
 varieties of dimension $\le 3$ rationally dominated by products of curves \cite[Example 3.15]{V3}, varieties $X$ with $A^i_{AJ}X_{}=0$ for all $i$ \cite[Theorem 4]{V2} (in particular, Fano 3--folds \cite{GoGu}), products of varieties with finite--dimensional motive \cite{Kim}.
\end{remark}

\begin{remark}
It is worth pointing out that up till now, all examples of finite-dimensional motives happen to be in the tensor subcategory generated by Chow motives of curves. On the other hand, ``many'' motives are known to lie outside this subcategory, e.g. the motive of a general hypersurface in $\PP^3$ \cite[Remark 2.34]{Ay}.
\end{remark}

\section{Bloch conjecture for some Calabi--Yau threefolds}

\begin{definition} Let $X$ be a Calabi--Yau threefold. A correspondence $\Gamma\in A^3(X\times X)$ is called {\em symplectic\/} if
  \[  \Gamma_\ast=\hbox{id}\colon\ \ H^{0,3}(X)\ \to\ H^{0,3}(X)\ .\]
  \end{definition}

\begin{theorem}\label{main1} Let $X$ be a Calabi--Yau threefold. Assume moreover

\noindent
{(\rom1)} $X$ has finite--dimensional motive;

\noindent
(\rom2) $B(X)$ is true;

\noindent
(\rom3) the generalized Hodge conjecture is true for $H^3(X)$.

Let $\Gamma\in A^3(X\times X)$ be a symplectic correspondence. Then
  \[ \Gamma_\ast\colon\ \ A^3_{hom}X\ \to\ A^3_{hom}X\ \]
  is an isomorphism.
  \end{theorem}
  
 \begin{remark} In case $X$ is not of general type (i.e., if we adhere to the usual definition of Calabi--Yau varieties), hypothesis (\rom2) is always fulfilled \cite{Tan}.
 \end{remark}

 \begin{proof} 
 Hypotheses (\rom1) and (\rom2) ensure the existence of a refined Chow--K\"unneth decomposition $\Pi_{i,j}$ as in \cite{V4}. There is a splitting 
   \[ H^3(X)= H^3_{\rm tr}(X)\oplus \wt{N}^1 H^3(X)\ ,\]
   where the ``transcendental cohomology'' $H^3_{\rm tr}(X)$ is defined as
   \[ H^3_{\rm tr}(X):=(\Pi_{3,0})_\ast H^3(X)\ \subset\ H^3(X)\ .\]
   
   Hypothesis (\rom3) implies that 
   \[  \bigl((\Gamma-\Delta)\circ \Pi_{3,0}\bigr)_\ast H^3(X)=0\ ,\]
   in view of lemma \ref{trans} below. This means that
  \[  \Gamma -\Delta = (\Gamma-\Delta)\circ ({\displaystyle \sum_{(i,j)\not=(0,3)} \Pi_{i,j}})\ \ \hbox{in\ } H^6(X\times X)\ .\]
  By construction of the $\Pi_{i,j}$, this implies
  \[ \Gamma-\Delta= R_0+R_1+R_2\ \ \hbox{in\ } H^6(X\times X)\ ,\]
  where $R_0, R_1, R_2$ are cycles supported on $(\hbox{point})\times X$, resp. on $(\hbox{divisor})\times (\hbox{divisor})$, resp. on $X\times(\hbox{point})$. That is, the cycle
  \[ \Gamma-\Delta-R_0-R_1-R_2\ \ \in A^3(X\times X)\]
  is homologically trivial. Applying the nilpotence theorem, and noting that the $R_\ell$ do not act on $A^3_{hom}X$, it follows that there exists $N\in\NN$ such that
    \[ (\Gamma^{\circ N})_\ast=\hbox{id}\colon\ \ A^3_{hom}X\ \to\ A^3_{hom}X\ .\]
    In particular, 
    \[ \Gamma_\ast\colon\ \ A^3_{hom}X\ \to\ A^3_{hom}X\ \]
is injective and surjective.

\begin{lemma}\label{trans} Let $X$ be a Calabi--Yau threefold, and assume the generalized Hodge conjecture is true for $H^3(X)$. Let $\Gamma \in A^3(X\times X)$ be a symplectic correspondence. Then
  \[ \Gamma_\ast=\hbox{id}\colon\ \ H^3_{\rm tr}(X)\ \to\ H^3_{\rm tr}(X)\ .\]
  \end{lemma}
  
 \begin{proof} The intersection pairing on $H^3(X)$ respects the decomposition
    \[  H^3(X)= H^3_{\rm tr}(X)\oplus \wt{N}^1 H^3(X)\ ,\]
i.e. restriction induces a non--degenerate pairing
  \[ \begin{split} H^3_{\rm tr}(X)\otimes H^3_{\rm tr}(X) \ &\to\ H^6(X)\ ,\\
                \end{split}\]
 and hence $H^3_{\rm tr}(X)$ and $\wt{N}^1 H^3(X)$ are orthogonal with respect to the intersection pairing.
                
  Let $\omega\in H^{0,3}(X)$ be a generator. By the truth of the generalized Hodge conjecture and remark \ref{is}, we have 
  \[  \wt{N}^1 H^3(X)=N^1 H^3(X)=\bigl\{  a\in H^3(X)\ \vert\ a_{\C} \cdot \omega=0\bigr\}\ . \]                
  
  Let $K\subset H^3(X)$ denote the kernel
     \[K:= \ker \Bigl( (\Gamma-\Delta)_\ast\colon H^3(X)\ \to\ H^3(X)\Bigr)\ .\]
     Since the correspondence $\Gamma$ is symplectic, we have (by definition)
    \[ H^{0,3}(X)\ \subset\ K_{\C}:=\ker \Bigl( (\Gamma-\Delta)_\ast\colon H^3(X,\C)\ \to\ H^3(X,\C)\Bigr)\ .\]
    But then,
    \[  K^\perp \ \subset\  \bigl\{  a\in H^3(X)\ \vert\ a_{\C} \cdot \omega=0\bigr\}=\wt{N}^1 H^3(X)\  \]
    (here ${}^\perp$ denotes the orthogonal complement with respect to the intersection pairing on $H^3(X)$).   
   This implies
   \[ K\ \supset\ \wt{N}^1 H^3(X)^\perp=H^3_{\rm tr}(X)\ .\]   
  \end{proof}
   \end{proof}        
   
 \begin{remark} Lemma \ref{trans} is inspired by the analogous result for $K3$ surfaces, which can be found in \cite[Proof of Corollary 3.11]{V8} or \cite[Lemma 2.5]{Ped}.
 \end{remark}  
  
\begin{remark} As for examples which satisfy the hypotheses of theorem \ref{main1}, all the examples of section \ref{examples} will do. Indeed, all examples in section \ref{examples} are rationally dominated by products of elliptic curves. As such, they have finite--dimensional motive and $B(X)$ is true. The generalized Hodge conjecture is true for products of elliptic curves \cite[Theorem 6.1]{Ab} (NB: for products of Fermat curves, it suffices to refer to \cite{Sh}); any blow--up of $E_1\times E_2\times E_3$ still satisfies the generalized Hodge conjecture in degree $3$, hence so do the Calabi--Yau varieties of section \ref{examples}, as they are dominated by such a blow--up.
\end{remark}  
          
   

\section{Indecomposability}

\begin{definition} Let $X$ be a smooth projective variety of dimension $n\le 5$. Assume $B(X)$ holds and $X$ has finite--dimensional motive. Then we define the ``transcendental motive'' $t(X)$
as
  \[ t(X):=(X,\Pi_{n,0},0)\ \ \in \MM_{\rm rat}\ ,\]
  where $\Pi_{n,0}$ is the refined Chow--K\"unneth projector constructed by Vial \cite[Theorem 2]{V4}.
  \end{definition}
  
\begin{remark} The fact that $t(X)$ is well--defined up to isomorphism follows from \cite[Theorem 7.7.3]{KMP} and \cite[Proposition 1.8]{V4}.
In case $n=2$, $t(X)$ coincides with the ``transcendental part'' $t_2(X)$ constructed for any surface in \cite{KMP}.
\end{remark}

\begin{theorem}\label{main2} Let $X$ be a Calabi--Yau threefold. Assume moreover

\noindent
{(\rom1)} $X$ has finite--dimensional motive;

\noindent
(\rom2) $B(X)$ is true;

\noindent
(\rom3) the generalized Hodge conjecture is true for $H^3(X)$.

Then $t(X)$ is indecomposable: any non--zero submotive of $t(X)$ coincides with $t(X)$.
\end{theorem}

\begin{proof} Suppose $V=(X,v,0)\subset t(X)$ is a submotive which is not the whole motive $t(X)$. Then in particular,
   \[ H^3(V)\ \subsetneqq\ H^3(t(X))=H^3_{\rm tr}(X)\ .\]
   (Indeed, suppose we have equality. Then $V=t(X)$ in $\MM_{\rm hom}$, and using finite--dimensionality this implies $V=t(X)$ in $\MM_{\rm rat}$, contradiction.)
 But $H^3_{\rm tr}(X)$ does not have non--trivial sub-Hodge structures: indeed, suppose $H^3(V,\C)$ contains $H^{3,0}(X)$. Then 
   \[ (v-c\Delta)_\ast H^{3,0}(X)=0\ ,\]
   for some non--zero $c\in \QQ$. But then, as in the proof of lemma \ref{trans},  
   \[  \Bigl(\ker\bigl( (v-c\Delta)\vert_{H^3}\bigr)\Bigr)^\perp\ \subset\ \wt{N}^1 H^3(X)\ ,\]
   whence
   \[   \ker\bigl( (v-c\Delta)\vert_{H^3}\bigr)\ \supset\ \wt{N}^1 H^3(X)^\perp= H^3_{\rm tr}(X)\ ;\]
   this is absurd as it contradicts the fact that $H^3(V)\not=H^3_{\rm tr}(X)$. Suppose next that $H^3(V,\C)$ does {\em not\/} contain $H^{3,0}(X)$, i.e. $v_\ast H^{3,0}(X)=0$. Then, again as in the proof of lemma \ref{trans}, we find that
   \[ \Bigl( \ker\bigl( v\vert_{H^3}\bigr)\Bigr)^\perp\ \subset\ \wt{N}^1 H^3(X)\ ,\]
   whence
   \[   \ker\bigl( v\vert_{H^3}\bigr)\ \supset\ \wt{N}^1 H^3(X)^\perp= H^3_{\rm tr}(X)\ ;\]  
   it follows that $H^\ast(V)=0$ and so (using finite--dimensionality) $V=0$ in $\MM_{\rm rat}$.    
\end{proof}


\begin{corollary}\label{aut} Let $X$ be as in theorem \ref{main2}. Let $G\subset\hbox{Aut}(X)$ be a finite group of finite order automorphisms.

\noindent
(\rom1) If $g\in G$ is symplectic, then
  \[   A^3_{hom}(X)=A^3_{hom}(Y)\ ,\]
  where $Y$ denotes a resolution of singularities of the quotient $X/G$.
  
 \noindent
 (\rom2) If $g\in G$ is not symplectic, then 
  \[ A^3_{hom}(Y)=0\ .\]
\end{corollary}

\begin{proof}

\noindent
(\rom1) After blowing up $X$ (which doesn't change $A^3$), we may assume the rational map $p\colon X\to Y$ is a morphism, i.e. $Y=X/G$.
The morphism $p$ induces a map of motives
  \[ p\colon\ \ t(X)\ \to\ t(Y)\ \ \hbox{in}\ \MM_{\rm rat}\ .\]
  Since
  \[  p_\ast p^\ast=s\cdot\hbox{id}\colon\ \ A^3_{hom}(Y)\ \to\ A^3_{hom}(Y)\ \]
  (where $s$ is the number of elements of $G$), this map of motives has a right--inverse (given by ${1/s}$ times the transpose of the graph of $p$). By general properties of pseudo--abelian categories, this means \cite[Remark 1.7]{Sch} that $t(Y)$ is (non--canonically) a direct summand of $t(X)$, i.e. we can write
  \[  t(X)=T_0\oplus T_1\ \ \hbox{in}\ \MM_{\rm rat}\ ,\]
  such that $p$ induces an isomorphism $T_0\cong t(Y)$. The motive $T_0$ cannot be $0$ (if it were $0$, then a fortiori $t(Y)\in\MM_{\rm hom}$ would be $0$ and hence $H^{3,0}(X)=H^{3,0}(Y)=0$, which is absurd). Applying theorem \ref{main2}, it follows that $T_0=t(X)$ and so
    \[ p\colon\ \ t(X)\ \xrightarrow{\cong}\ t(Y)\ \ \hbox{in}\ \MM_{\rm rat}\ .\]
    
  \noindent
  (\rom2) As in the proof of (\rom1), we have a splitting
  \[ t(X)=T_0\oplus T_1\ \ \hbox{in}\ \MM_{\rm rat}\ ,\]
  such that $p$ restricts to an isomorphism $T_0\cong t(Y)$. The motive $T_0$ cannot be all of $t(X)$ (if it were, then also $p\colon t(X)\cong t(Y)$ in $\MM_{\rm hom}$ and hence $H^{3,0}(X)\cong H^{3,0}(Y)$. But this is absurd, for the projector ${1\over s}\sum_{g\in G} \Gamma_g$ acts as $0$ on $H^{3,0}(X)$). It follows that $T_0=0$ and so
  \[  t(Y)=0\ \ \hbox{in}\ \MM_{\rm rat}\ .\]
 \end{proof}

\section{Voisin's conjecture}

\begin{conjecture}[Voisin \cite{V9}]\label{voisin} Let $X$ be a Calabi--Yau threefold. For any $k\ge 2$, let the symmetric group $S_k$ act on $X^k$ by permutation of the factors.
Let $pr_k\colon X^k\to X^{k-1}$ denote the projection obtained by omitting one of the factors. The induced map 
  \[  (pr_k)_\ast\colon\ \ A_j(X^k)^{S_k}\ \to\ A_j(X^{k-1})\]
  is injective for $j\le k-2$.
\end{conjecture}

\begin{remark} Suppose $X$ has a Chow--K\"unneth decomposition $h(X)=\sum_i h^i(X)$ in $\MM_{\rm rat}$. Then conjecture \ref{voisin} is equivalent to the following: for any $k\ge 2$, the Chow motive $\hbox{Sym}^k h^3(X)$ satisfies
  \[ A_j \bigl( \hbox{Sym}^k h^3(X)\bigr)=0\ \ \hbox{for\ all\ }j\le k-2\ .\]
  In case $k=2$, conjecture \ref{voisin} predicts the following concrete statement about $0$--cycles: let $a,a^\prime\in A^3_{hom}(X)$ be two $0$--cycles of degree $0$. Then
  \[  a\times a^\prime= - a^\prime\times a\ \ \hbox{in}\ A^6(X\times X)\ .\]
 \end{remark}
 
 \begin{remark} A conjecture similar to conjecture \ref{voisin} can be formulated for varieties of geometric genus $1$ in any dimension. We refer to \cite{V9} and \cite[Conjecture 4.37 and Example 4.40]{Vo} for precise statements, and verifications in certain cases.
 \end{remark}

\begin{theorem}\label{main3} Let $X$ be a Calabi--Yau threefold. Assume moreover

\noindent
(\rom1) $X$ has finite--dimensional motive;

\noindent
(\rom2) $B(X)$ is true;

\noindent
(\rom3) $X$ is {\em rigid\/}, i.e. $h^{2,1}(X)=0$.

Then conjecture \ref{voisin} is true for $X$.

\end{theorem}

\begin{proof} Hypotheses (\rom1) and (\rom2) ensure the existence of a Chow--K\"unneth decomposition $\Pi_i$, i.e.
  \[ h(X)= h^0(X)\oplus \cdots \oplus h^6(X)\ \ \hbox{in}\ \MM_{\rm rat}\ ,\]
  where $h^i(X)=(X,\Pi_i,0)$.
 Let
  \[ \Lambda_k:= {1\over k!}\Bigl(\displaystyle \sum_{\sigma\in S_k} \Gamma_\sigma\Bigr)\circ (\Pi_3^{\otimes k})\ \ \in A^{3k}(X^k\times X^k)\ .\]
 On the level of cohomology, the correspondence $\Lambda_k$ is a projector on $\hbox{Sym}^k H^3(X)\subset H^{3k}(X^k)$; on Chow--theoretical level $\Lambda_k$ is idempotent and defines
 the Chow motive $\hbox{Sym}^k h^3(X)$ in the language of \cite{Kim}. 
 
 Hypothesis (\rom3) implies that $\dim H^3(X)=2$, hence for $k\ge 3$ one has
   \[ \Lambda_k=0\ \ \hbox{in}\ H^{6k}(X^k\times X^k)\ .\]
   Using the nilpotence theorem, it follows that
   \[ \Lambda_k=0\ \ \hbox{in}\ A^{3k}(X^k\times X^k)\ .\]
   
   It only remains to check the case $k=2$. Note that $\hbox{Sym}^2 H^3(X)$ has dimension $1$, and
     \[  \hbox{Sym}^2 H^3(X)\ \subset\ H^6(X^2)\cap F^3\ .\]
     What's more, the Hodge conjecture is true for this subspace, since
     \[   \hbox{Sym}^2 H^3(X)=\QQ\cdot \Pi_3 \ \subset\ H^6(X^2)\ .\]
    It follows that 
    \[ \Lambda_2= \Pi_3\times \Pi_3\ \ \hbox{in}\ H^{12}(X^2\times X^2)\ ,\]
    and hence (using the nilpotence theorem)
    \[ \Lambda_2= \Pi_3\times \Pi_3\ \ \hbox{in}\ A^{6}(X^2\times X^2)\ .\]     
   It follows that
    \[  (\Lambda_2)_\ast \bigl(A_j(X^2)\bigr) =0\ \ \hbox{for\ all}\ j\le 2\ ,\]
    i.e. a strong form of Voisin's conjecture is true.
  \end{proof}

\begin{remark} Theorem \ref{main3} applies to the examples in subsection \ref{rigid}, and also to the two examples of remark \ref{more}.
\end{remark}

\begin{remark} 
In the proof of theorem \ref{main3}, we have used the condition $\dim H^3(X)=2$, which is a consequence of hypothesis (\rom3). By replacing in the proof the correspondence $\Pi_3$ by $\Pi_{3,0}$ (i.e., replacing the motive $h^3(X)$ by $t(X)$), it is enough to assume 
  \[ \dim H^3_{tr}(X)=2\ ,\]
  a condition a priori weaker than (\rom3).

\end{remark}


\begin{acknowledgements} The ideas developed in this note grew into being during the Strasbourg 2014---2015 groupe de travail based on the monograph \cite{Vo}. Thanks to all the participants of this groupe de travail for a pleasant and stimulating atmosphere. Thanks to Charles Vial for interesting email correspondence, and to Bert van Geemen for informing me about the Beauville threefold and providing the reference \cite{FG}. Thanks to the referee for helpful remarks on a prior version.
Many thanks to Yasuyo, Kai and Len for making it possible to work at home in Schiltigheim.
\end{acknowledgements}


\end{document}